\documentclass[11pt, reqno]{amsart}
\usepackage{amscd,amsfonts,amsmath,amssymb,amstext,amsthm}
\usepackage[all]{xy}
\usepackage{tikz-cd}
\usepackage{tikz}
\usetikzlibrary{calc}
\usepackage{xcolor}
\usepackage{cancel}
\usepackage{verbatim}
\usepackage{enumitem}
\usepackage{graphicx}
\usepackage{hyperref}
\usepackage{cleveref}
\usepackage{subcaption}

\theoremstyle{plain}
\newtheorem{theorem} {Theorem} [section]
\newtheorem{lemma} [theorem]{Lemma}
\newtheorem{proposition}[theorem]{Proposition}
\newtheorem{corollary} [theorem]{Corollary}

\newtheorem{algorithm}[theorem]{Algorithm}
\theoremstyle{definition}
\newtheorem{definition}[theorem]{Definition}
\newtheorem{example}[theorem]{Example}

\newtheorem{remark}[theorem]{Remark}

\numberwithin{equation}{section}

\title[Ricci curvatures of Graphs]{A combinatorial approach to $\alpha$-Ricci and Lin-Lu-Yau Ricci curvatures on Graphs}
\date\today
\author{Jaewoo Jung, Young Soo Kwon, Seungjae Lee, and Jihye Park}

\address{Global Basic Research Laboratory - Algebra and Geometry of Spaces of Tensors, and Applications (GBRL-AGSTA), Daegu Gyeongbuk Institute of Science and Technology (DGIST), 333 Techno Jungang-daero, Hyeonpung-eup, Dalseong-gun, Daegu 42988, Republic of Korea}
\email{jaewoojung@dgist.ac.kr}

\address{Yeungnam University
Department of Mathematics, 280 Daehak-Ro, Gyeongsan, Gyeongbuk 38541, Republic of Korea}
\email{ysookwon@ynu.ac.kr}

\address{Kyungpook National University,
80 Daehak-ro, Buk-gu, Daegu, Gyeongbuk 41566, Republic of Korea}
\email{seungjae@knu.ac.kr}

\address{Yeungnam University
Department of Mathematics, 280 Daehak-Ro, Gyeongsan, Gyeongbuk 38541, Republic of Korea}
\email{jihyepark@ynu.ac.kr}

\subjclass[2020]{Primary 51F30; Secondary 05C12, 49Q22}
\keywords{Discrete Ricci curvature, Ollivier Ricci curvature, Lin--Lu--Yau Ricci curvature, Wasserstein distance, Optimal transport, Graph matching.}
\thanks{{Authors equally contributed to this work.}}

\begin{document}

\begin{abstract}
In this paper, we study the $\alpha$-Ricci curvature and the Lin--Lu--Yau Ricci curvature on simple, connected, and locally finite graphs. For regular graphs, we introduce a combinatorial construction of optimal transport plans realizing the $1$-Wasserstein distance and use it to derive exact formulas for the $\alpha$-Ricci curvature and the Lin--Lu--Yau Ricci curvature. This yields a combinatorial proof of the known curvature formulas. Furthermore, for non-regular graphs, we characterize conditions on the size of common neighborhoods that guarantee either non-negative or vanishing Lin--Lu--Yau Ricci curvature.
\end{abstract}

\maketitle

\section{Introduction}

Ricci curvature is fundamental to Riemannian geometry, whose positivity governs the local geometry and global structure of manifolds.
Various discretizations of Ricci curvature on graphs have been proposed, including those of Bakry and Émery \cite{BE85}, Chung and Yau \cite{CY95}, Ollivier \cite{Oli}, and Lin, Lu, and Yau \cite{LLY2011}. We focus on the notions introduced by Ollivier and by Lin, Lu, and Yau. 
Ollivier \cite{Oli} defined the $\alpha$–Ricci curvature between graph vertices using Markov chains induced by lazy random walks, while Lin, Lu, and Yau \cite{LLY2011} introduced the Lin–Lu–Yau Ricci curvature as a finer approximation to manifold Ricci curvature.

Calculations of the $\alpha$–Ricci curvature and the Lin–Lu–Yau Ricci curvature have attracted attention. Jost and Liu \cite{JL} proved an explicit formula for the $0$–Ricci curvature on trees. Bhattacharya and Mukherjee \cite{BM} established an exact formula for the $0$–Ricci curvature on bipartite graphs and graphs of girth at least $5$, and obtained a lower bound for graphs of girth $4$. More recently, Hehl \cite{Hehl25, Hehl26} established explicit formulas for the $\alpha$–Ricci curvature and the Lin–Lu–Yau Ricci curvature on locally finite graphs for adjacent vertices of equal finite degree.

One of the main ingredients of this line of research is to determine explicit values or ranges of the $1$-Wasserstein distance, either by constructing explicit $1$-Lipschitz functions or by applying results from optimal transport and linear programming. In this paper, we take a different perspective and develop a combinatorial approach for regular graphs based on the neighborhood structure of adjacent vertices.

\begin{theorem}\label{thm:main}
Let $G=(V,E)$ be a simple connected locally finite regular graph and
$\{x,y\}\in E$ with $\deg_G(x)=\deg_G(y)=d$. 
Let $A=N_G(x)\setminus N_G[y], B=N_G(y)\setminus N_G[x]$ and let $K$ be the complete bipartite graph with partite sets $A$ and $B$.  
Then $K$ admits a perfect matching
$M=M_1\cup M_2\cup M_3$ with
$M_i=\{\{u,v\}\in M:\operatorname{dist}_G(u,v)=i\},
i=1,2,3$ satisfying
\[
\frac{\kappa^\alpha(x,y)}{1-\alpha}
=
\begin{cases}
\dfrac{d+1-(3|A|-2|M_1|-|M_2|)}{d},
& \alpha\ge \dfrac1{d+1},\\[1.2ex]
\dfrac{d-(2-t)-(3|A|-2|M_1|-|M_2|)}{d}
+\dfrac{\alpha(3-t)}{1-\alpha},
& \alpha<\dfrac1{d+1},
\end{cases}
\]
where $\kappa^\alpha(x,y)$ is the $\alpha$-Ricci curvature and $t=\max_{\{u,v\}\in E(M)}\operatorname{dist}_G(u,v)
$.
\end{theorem}

To prove Theorem~\ref{thm:main}, we introduce Algorithm~\ref{alg:regular}, which explicitly constructs an optimal transport plan realizing the $1$-Wasserstein distance. This yields a combinatorial proof of explicit formulas of Hehl \cite{Hehl25, Hehl26} for the $\alpha$-Ricci and Lin--Lu--Yau Ricci curvatures, avoiding the use of Choquet’s minimization theorem and Birkhoff’s theorem.
In addition, Theorem \ref{thm:main} covers the range $\alpha<1/(d+1)$, which was not treated in those works.

In addition, for non-regular graphs we characterized the common neighborhood size that guarantees either non-negative or vanishing Lin-Lu-Yau Ricci curvature. 
\medskip
\begin{theorem}\label{thm:LLY_nonnegative}
Let $G=(V,E)$ be a simple, connected, and locally finite graph. 
For an edge $\{x, y\} \in E$, let $d_x:=\deg_G(x)$ and $d_y:=\deg_G(y)$ with $d_x < d_y$. 
Define $A := N_G(x) \setminus N_G[y]$, $B := N_G(y) \setminus N_G[x]$, $\Delta := N_G(x) \cap N_G(y)$, and let $\delta := |\Delta|$.
If
$$
\delta  \geq  \frac{2d_x d_y - 2d_x - 2d_y}{d_y + 2d_x},
$$ 
then the Lin-Lu-Yau Ricci curvature  $\kappa^{\textup{LLY}}(x,y) \geq 0$. Moreover, under the assumption $\delta \geq \frac{2d_x d_y - 2d_x - 2d_y}{d_y + 2d_x}$,  $\kappa^{\textup{LLY}}(x,y)  = 0$ if and only if the following conditions are simultaneously satisfied:
    \begin{enumerate}
        \item 
        $
        \delta = \frac{2d_x d_y - 2d_x - 2d_y}{d_y + 2d_x}
        $
        \item $E({\Delta}, B) = \emptyset$
        \item $\operatorname{dist}_G(u,w) = 3$ for all $u \in A$ and $w \in B$
    \end{enumerate}    
where {$E(A,B)$ is the set of edges between $A$ and $B$.}  
\end{theorem}
\medskip
Moreover, in Theorem \ref{thm:CurvatureRange}, we show that previous results on the upper and lower bounds of the Lin–Lu–Yau Ricci curvature, including \cite{CKKLMP20} and \cite{LLY14}, extend to the $\alpha$-Ricci curvature for arbitrary $\alpha$. This theorem not only recovers the known bounds for the Lin–Lu–Yau Ricci curvature, but also generalizes them to the $\alpha$-Ricci curvature on arbitrary graphs without additional assumptions. 

We also study $\alpha$-Ricci curvatures of graph joins. In \cite{LLY2011,CKKLP20}, curvature properties of natural graph products, such as tensor, Cartesian, and strong products, are investigated, and conditions for positivity are established. These products preserve regularity, whereas graph joins typically produce non-regular graphs. In Proposition \ref{prop:graphjoins}, we analyze the $\alpha$-Ricci and Lin–Lu–Yau Ricci curvatures of graph joins.

The paper is organized as follows. Section \ref{sec:prelim} reviews preliminaries on graph theory and $\alpha$-Ricci and Lin--Lu--Yau Ricci curvature. Section \ref{sec:regular} introduces Algorithm \ref{alg:regular} and proves Theorem \ref{thm:main}. Section \ref{sec:algorithm} presents Algorithm \ref{alg:regular} for transport plans between probability distributions with finitely supported difference. Section \ref{sec:range} proves Theorem \ref{thm:LLY_nonnegative} and establishes upper and lower bounds for the $\alpha$-Ricci curvature. Section \ref{sec:join} determines the exact value and a lower bound for the $\alpha$-Ricci curvature of graph joins.

\section{Preliminaries}\label{sec:prelim}

In this section, we introduce the notation and definitions used throughout the paper. 

\subsection{Graph Theory}

In this subsection, we recall some notions from graph theory.
For further background, we refer the reader to \cite{diestel2012graph}.

A graph $G=(V,E)$ is called simple if it is undirected and has neither loop nor multiple edge.
We write $V(G):=V$ and $E(G):=E$.
For $x \in V(G)$, the \textit{open neighborhood of $x$ in $G$} is defined by $N_G(x):=\{y\in V \mid \{ x, y \} \in E(G)\}$,
and the \textit{closed neighborhood of $x$ in $G$} is $N_G[x]:=N_G(x)\cup \{x \}$. 
The \textit{degree} $\deg_G(x)$  of $x$ is defined as the size $|N_G(x)|$ of $N_G(x)$.
We denote the distance between vertices $x, y \in V$ by $\operatorname{dist}_G(x,y)$. 

For a partition $V(G)=U_1\sqcup \cdots \sqcup  U_\ell$, we denote by $K_{U_1,\cdots, U_\ell}$ the complete $\ell$-partite graph with partite sets $U_1, \cdots, U_\ell$. 
We simply call a complete 2-partite graph a complete bipartite graph. 
When the partition of the vertex set is clear from the context, we simply write $K$.

A \textit{matching} $M$ of $G$ is a subgraph of $G$ whose edge set has the property that no two edges share a common vertex, and whose vertex set consists precisely of the endpoints of the edges of $M$.
A matching $M$ is called \textit{maximum} if $|E(M)|$ is {maximum} among all matchings in $G$. 
A matching $M$ is called \emph{perfect} if every vertex of $G$ is an endpoint of exactly one edge in $M$. 

In what follows, all graphs are assumed to be simple and connected.

\subsection{$\alpha$-Ricci and Lin--Lu--Yau Ricci curvature}
In this subsection, we introduce the $\alpha$-Ricci and Lin-Lu-Yau Ricci curvature on a graph.
Let $G=(V,E)$ be a graph. A \textit{probability distribution} is a mapping $\tau: V \rightarrow [0,1]$ such that $\sum_{x \in V} \tau(x) =1$. Suppose that two probability distributions  $\mu$ and $\nu$ on  $V$ have finite support.
A \textit{transport plan} $\pi$ transporting $\mu$ to $\nu$ is a function $\pi: V \times V \rightarrow [0,1]$ satisfying 
$$
\sum_{y \in V} \pi (x,y) = \mu(x), \quad \sum_{x \in V} \pi(x,y) = \nu(y).
$$
The 1-\textit{Wasserstein distance} between $\mu$ and $\nu$ is defined by
$$
W_1(\mu, \nu) =  \inf_{\pi \in \Pi(\mu,\nu)} \sum_{x ,y \in V} \pi(x,y) \cdot \text{dist}_{G}(x,y)
$$
where $\Pi(\mu, \nu)$ is the set of all transport plans transporting $\mu$ to $\nu$. We say that $f:V \rightarrow \mathbb{R}$ is \textit{$M$-Lipschitz} if 
\begin{equation*}
|f(x)-f(y)| \leq M \cdot \text{dist}_{G}(x,y) 
\end{equation*} 
for any $x,y \in V$. 

The Kantorovich--Rubinstein duality provides a convenient way to obtain lower bounds for $W_1$ using $1$-Lipschitz functions on $V$. 

\begin{theorem}[Kantorovich--Rubinstein Duality Theorem]
\label{thm:KR_duality}
Let $G=(V,E)$ be a graph, and $\mu, \nu$ be two probability distributions {with finite support} on $V$.
Then,
$$
W_1(\mu, \nu) =  \sup_{f \in \textup{Lip}_1(V)} \sum_{v \in V} f(v) \cdot \big( \mu(v) - \nu(v)\big) 
$$
where $\textup{Lip}_1(V)$ is the set of all 1-Lipschitz functions on $V$.
\end{theorem}
\medskip
For a proof of Theorem \ref{thm:KR_duality}, we refer the reader to \cite[Chapter 6]{villani}. Now, we define the probability distribution $m^\alpha_x$ for $x \in V$ and $\alpha \in [0,1]$ by

\begin{equation}\label{mass funcion for alpha-Ric}
m_x^{\alpha}(v) = 
\begin{cases}
\alpha & \text{if } v = x, \\
\displaystyle\frac{1 - \alpha}{\deg_G(x)} & \text{if } v \in N_G(x), \\
0 & \text{if} \; v \in V \setminus N_G[x].
\end{cases}
\end{equation} 
In \cite{Oli}, Ollivier introduced the concept of $\alpha$-Ricci curvature. 

\begin{definition}[$\alpha$-Ricci curvature] Let $G=(V,E)$ be a locally finite graph. {For two vertices  $x,y \in V,$} $\alpha$-Ricci curvature $\kappa^{\alpha}(x,y)$ is defined by 
$$
\kappa^{\alpha} (x,y) = 1- \frac{W_1(m_x^{\alpha}, m_y^{\alpha})}{\text{dist}_G (x,y)}.
$$ 
\end{definition}

In \cite{LLY2011}, Lin, Lu, and Yau proved that for any $\alpha \in [0,1]$, the function { $\kappa^{\alpha}(x,y)$ is concave in $\alpha$ and that $\frac{\kappa^{\alpha}(x,y)}{1-\alpha}$}
is uniformly upper bounded in $\alpha$. Motivated by these facts, they introduce the Lin--Lu-Yau Ricci curvature. 

\begin{definition}[Lin--Lu--Yau Ricci curvature]\label{LLY Ricci}

Let $G=(V,E)$ be a locally finite graph. The \textit{Lin-Lu-Yau Ricci curvature} $\kappa^{\text{LLY}}(x,y)$ of {two vertices $x, y \in V$} is defined by 
\begin{equation}
\begin{aligned}
\kappa^{\text{LLY}}(x,y)= \lim_{\alpha \rightarrow 1^-} \frac{\kappa^\alpha (x,y)}{1-\alpha}. 
\end{aligned}
\end{equation}
\end{definition}

\section{Alpha-Ricci Curvature of Regular Graphs}\label{sec:regular}

In this section, we investigate the $\alpha$-Ricci curvatures of $d$-regular graphs. 
We show that the curvature is determined by the size of the maximum matching in the local neighborhood. 

Let $G$ be a $d$-regular graph. For an edge $\{x,y \} \in E(G)$, let $A := N_G(x) \setminus N_G[y]$ and $B := N_G(y) \setminus N_G[x]$.
Let $K$ %$K_{A,B}$ 
be the complete bipartite graph whose two partite sets are $A$ and $B$.  
Note that for any $u \in A$ and $v \in B$, the distance $\operatorname{dist}_G(u, v)$ is at most $3$.
For any $i=1,2,3$, let $E_i = \{ \{u,v \} \ | \ u \in A, v \in B, \operatorname{dist}_G(u, v) = i \}$ and let  $H_i$ be the subgraph of $K$ induced by $E_i$. Let $H$ be the subgraph of $K$ induced by $E_1 \cup E_2$. 

The following is an algorithm to find some transportation plan from $m^\alpha_x$ to $m^\alpha_y$, which will be shown to be optimal. Note that the algorithm also works well when degrees of $x$ and $y$ are equal, even if $G$ is not regular.

\begin{algorithm}[{Transportation plan from $m^\alpha_x$ to $m^\alpha_y$}
]\label{alg:regular}
\phantom{}
\begin{enumerate}
    \item[1.] Find a maximum matching $M_1$ of $H_1$. 
    Let $\mathcal{M}$ be the set of all maximum matchings of $H_1$.
    \item[2.] For $M_1 \in \mathcal{M}$, find a maximum matching $M_2$ of $H -V(M_1)$.
    \item[3.] Choose $M_1 \in \mathcal{M}$ such that $2|E(M_1)| + |E(M_2)|$ is maximized. 
    Let $M$ be the union of such $M_1$ and $M_2$.
    \item[4.] Find a perfect matching $M_3$ in $K-V(M)$. 
    \item[5.]
\begin{enumerate}
    \item[(1)] If $\alpha \ge \frac{1}{d+1}$, then we transport $\alpha - \frac{1-\alpha}{d}$ from $x$ to $y$ and for any edge $\{u, v \} \in E(M \cup M_3)$ with $u \in A$ and $v \in B$, transport $\frac{1-\alpha}{d}$ from $u$ to $v$.   
    \item[(2)] If $\alpha < \frac{1}{d+1}$, then choose an edge $\{u_1, v_1 \}$ in $M \cup M_3$ such that 
    \begin{equation*}
        \operatorname{dist}_G(u_1, v_1) =\max_{\{u,v\}\in E(M \cup M_3)}\operatorname{dist}_G(u,v).
    \end{equation*}
    Let $u_1 \in A$ and $v_1 \in B$. 
    Transport $\frac{1-\alpha}{d} - \alpha$ from $u_1$ to $x$; and from $y$ to $v_1$. Transport $\alpha$ from $u_1$ to $v_1$ and for any edge $\{u, v \} \in E(M \cup M_3)$ with $u \in A \setminus \{ u_1 \}$ and $v \in B \setminus \{v_1 \}$,  transport $\frac{1-\alpha}{d}$ from $u$ to $v$.
\end{enumerate}
\end{enumerate}
\end{algorithm}

\begin{remark}\label{algorithm T for lower}
Note that for $0 \le \alpha < \frac{1}{d +1}$,  $\alpha- \frac{1-\alpha}{d}<0$.
However, the optimal transportation plan for this case also depends on the matchings $M_1$ and $M_2$ obtained in Step 1, 2, and 3. Thus, the sizes $|E(M_1)|$ and $|E(M_2)|$ determine $W_1(m_x^\alpha, m_y^\alpha)$ for all $\alpha \in [0,1]$. We calculate $W_1(m_x^\alpha, m_y^\alpha)$ for all $\alpha \in [0,1]$ in the proof of Theorem \ref{thm:regularalg} 
\end{remark}

Let the transportation plan obtained by Algorithm \ref{alg:regular} 
be denoted by $\pi_T$.   
 For a transportation plan $\pi$ from $m^\alpha_x$ to $m^\alpha_y$, let 
$$ E_{\pi,i} = \{ \{u,v \} \ | \ \pi(u,v) >0 , \ u \in A, v \in B, \ \operatorname{dist}_G(u, v) = i \} $$
for any $i=1,2,3$.  Let $H_{\pi}$ be the subgraph of $K$ induced by $E_{\pi,1} \cup E_{\pi,2}$. 
The following theorem shows that $\pi_T$ is an optimal transportation plan  from $m^\alpha_x$ to $m^\alpha_y$.

\begin{theorem}\label{thm:regularalg}
The transportation plan $\pi_T$ obtained by Algorithm \ref{alg:regular} is optimal.
\end{theorem}

\begin{proof}
For the first case, assume that $\alpha \ge \frac{1}{d+1}$. Now one can assume that any optimal transportation plan  transports a mass of $\alpha-\frac{1-\alpha}{d}$ from $x$ to $y$ by \cite[Lemma 3.2]{Hehl26}. 
Suppose that $\pi_T$ is not optimal. 
Let $\pi$ be an optimal transportation plan such that the sum $\sum_{e \in E(M_1)} \pi(e)$ is maximized. 
Here, by a slight abuse of notation, we write $\pi(e) := \pi(u,v)$ for an edge $e = \{u,v\}$ with $u\in A$ and $v\in B$. 
Let $H_{M,\pi}$ be the subgraph of $K$ induced by $E(M) \cup E_{\pi,1} \cup E_{\pi,2}$. Since $\pi$ is optimal and $\pi_T$ is not optimal, there exists a connected component $S$ of $H_{M,\pi}$ such that
$$2\sum_{e \in E(S) \cap E_{\pi,1}} \pi(e) + \sum_{e \in E(S) \cap E_{\pi,2}} \pi(e) > \dfrac{1 - \alpha}{d}\left( 2|E(S) \cap E(M_1)| + |E(S) \cap E(M_2)|\right).$$

Since $M_1$ is a maximum matching of $H_1$, the matching $E(S)  \cap E(M_1)$ is a maximum matching of the subgraph of $S$ induced by $E(S) \cap E_{\pi,1}$. Since the number of edges in maximum matching equals to the order of the minimum covering in a bipartite graph, 
\begin{equation*}
    \sum_{e \in E(S) \cap E_{\pi,1}} \pi(e)  \le \dfrac{1 - \alpha}{d}|E(S) \cap E(M_1)|.
\end{equation*}

\textbf{Case 1--1: $A \cap V(S) = A \cap V(S)\cap V(M)$ or $B \cap V(S) = B \cap V(S) \cap  V(M)$.}

Without loss of generality, assume that $A \cap V(S) = A \cap V(S) \cap V(M)$. 
 Since $E(S) \cap E(M)$ is a maximum matching of $S$, we have 
$$\sum_{e \in E(S) \cap (E_{\pi,1}\cup E_{\pi,2})} \pi(e)  \le \dfrac{1 - \alpha}{d}|E(S) \cap E(M)| =\dfrac{1 - \alpha}{d} |A \cap V(S) |.$$
 This implies that 
$$2\sum_{e \in E(S) \cap E_{\pi,1}} \pi(e) + \sum_{e \in E(S) \cap E_{\pi,2}} \pi(e)  \le\dfrac{1 - \alpha}{d} \left( 2|E(S) \cap E(M_1)| + |E(S) \cap E(M_2)|\right),$$
which is a contradiction.

\medskip

\textbf{Case 1--2: $A \cap V(S) \neq A \cap V(S) \cap  V(M)$ and $B \cap V(S) \neq B \cap V(S) \cap V(M)$.}

Suppose that there is no $M$-augmenting path in $S$. 
Then $E(S) \cap E(M)$ is a maximum matching of $S$, and hence there is a covering set $C$ of $S$ such that $|C|=|E(S) \cap E(M)|$.
This implies that  $$2\sum_{e \in E(S) \cap E_{\pi,1}} \pi(e) + \sum_{e \in E(S) \cap E_{\pi,2}} \pi(e)  \le\dfrac{1 - \alpha}{d} \left( 2|E(S) \cap E(M_1)| + |E(S) \cap E(M_2)|\right),$$
which is a contradiction.
So there is an $M$-augmenting path in $S$. Let $Q:u_0, v_1, u_1, \cdots, u_j,v_{j+1}$ be a shortest $M$-augmenting path in $S$. Let $a =|(E(Q)-E(M))\cap E_1|$ and  $b =|(E(Q) \cap E(M))\cap E_1|$. 

If $a \ge b$, then for any $i=1,2$, let $\widetilde{M_i}$ be a matching composed of edges in $(E(M_i) -E(Q)) \cup \left((E(Q)-E(M))\cap E_i \right)$ and  $\widetilde{M} = M \vartriangle Q$.  
Now $\widetilde{M_1}$ and 
$\widetilde{M_2}$ satisfy the conditions 1 and 2 in Algorithm  \ref{alg:regular} and  $2|E(\widetilde{M_1})| + |E(\widetilde{M_2})| > 2|E(M_1)| + |E(M_2)|$, which contradicts the choice of $M$. So we have $a \le b-1$.

Let $\beta = \min \{ \pi(u_i, v_{i+1})) \ | \ i=0,1,\ldots, j \}$. Note that $0 < \beta < \dfrac{1 - \alpha}{d}$. Let $\widetilde{\pi}$ be a transportation plan from $m^\alpha_x$ to $m^\alpha_y$ defined by
\[
\widetilde{\pi}(u,v) =
\begin{cases}
    \pi(u,v) & \text{if } \{u,v \} \notin E(Q) \; \mbox{and} \; \{u,v\} \neq \{u_0, v_{j+1} \} , \\[4pt]
    \pi(u,v) - \beta & \text{if } \{u,v \}  \in E(Q) -E(M), \\[4pt]
    \pi(u,v) + \beta & \text{if } \{u,v \} \in E(Q) \cap E(M), \\[4pt]
    \pi(u,v) + \beta & \text{if } \{u,v \} = \{u_0, v_{j+1} \}.
\end{cases}
\]
Now the total cost of $\widetilde{\pi}$ minus the total cost of $\pi$ is at most
\[ 
\beta (b-a) +2(a-b-1) +3 = \beta (a-b+1) \le 0. 
\]
This implies that  $\widetilde{\pi}$ is also an optimal transportation plan and 
\[
\sum_{e \in E(M_1)} \widetilde{\pi}(e) > \sum_{e \in E(M_1)}  \pi(e), 
\]
which contradicts the choice of $\pi$. Therefore, $\pi_T$ is optimal. 

 Now we consider the case where $\alpha < \frac{1}{d+1}$.
 Note that in this case $x$ should receive $\frac{1-\alpha}{d} -\alpha$ and $y$ should send $\frac{1-\alpha}{d} -\alpha$. 
 Let $\pi'$ be an optimal transportation plan from $m^\alpha_x$ to $m^\alpha_y$ in this case. Suppose that $x$ receive $\gamma$ from $A$ and $\frac{1-\alpha}{d} -\alpha - \gamma$ from $y$ under the plan $\pi'$. This implies that $y$ send $\gamma$ to $B$. 
 Note that if $\gamma$ is $0$, then one can show that
 $$ \sum_{i=1}^{3} \sum_{e \in E(H_i)} i \pi'(e) = \sum_{i=1}^{3} \frac{1-\alpha}{d} \times i|M_i|$$
by a similar way with Case 1--1 and Case 1--2.

\medskip
\textbf{Case 2--1: $M_3 \neq \emptyset$.}

Let $\sum_{e \in E(H)} \pi'(e) = \frac{1-\alpha}{d} |M| - \gamma'$.  Then $\sum_{e \in E(H_3)} \pi'(e) = \frac{1-\alpha}{d} |M_3| - \gamma + \gamma'.$ This implies that 
 $$ \sum_{i=1}^{3} \sum_{e \in E(H_i)} i \pi'(e) \ge \sum_{i=1}^{3} \frac{1-\alpha}{d} \times i|M_i| - 3 \gamma$$
 and 
 \begin{eqnarray*} W_1(m_x^\alpha, m_y^\alpha)  &\ge & 2\gamma +\frac{1-\alpha}{d} -\alpha - \gamma +  \frac{1-\alpha}{d} \left( 3|A| - 2|M_{1}| - |M_{2}| \right) -3\gamma   \\ 
&=& -2 \gamma  -\alpha +  \frac{1-\alpha}{d} \left(1+ 3|A| - 2|M_{1}| - |M_{2}| \right).
\end{eqnarray*}
Since $0 \le \gamma \le \frac{1-\alpha}{d} -\alpha$, we have $$ W_1(m_x^\alpha, m_y^\alpha)  \ge  -2 \left( \frac{1-\alpha}{d} -\alpha \right)  -\alpha +  \frac{1-\alpha}{d} \left(1+ 3|A| - 2|M_{1}| - |M_{2}| \right).$$
Since $\pi_T$ obtained by Algorithm \ref{alg:regular} gives the lower bound of  $W_1(m_x^\alpha, m_y^\alpha)$,  $\pi_T$ is optimal.

\medskip

\textbf{Case 2--2: $M_2 \neq \emptyset$ and $M_3 = \emptyset$.}

Let $\tilde{G}$ be a graph obtained from $G$ by adding a new vertex $w$ and all edges between $w$ and vertices in $A \cup B$. Let  the 1-Wasserstein distance in $\tilde{G}$ between $m_x^\alpha$ and $ m_y^\alpha$ be denoted by  $W'_1(m_x^\alpha, m_y^\alpha)$.
Note that $W_1(m_x^\alpha, m_y^\alpha) \ge W'_1(m_x^\alpha, m_y^\alpha)$ and for any $u \in A$ and $v \in B$ such that $\{u,v \} \notin E(H_1)$, $\operatorname{dist}_{\tilde{G}}(u, v)=2$.

Now one can assume that  $M_1$ and $M_2$ obtained by applying  Algorithm \ref{alg:regular} with $\tilde{G}$ and $G$ are the same  because $$\{ (u, v) \ | \ u \in A, v \in B, \ \operatorname{dist}_{\tilde{G}}(u, v)=1 \} = \{ (u, v) \ | \ u \in A, v \in B, \ \operatorname{dist}_G(u, v)=1 \}$$ and $M_3 = \emptyset$.

By a similar way with Case 2--1, one can show that 
\begin{eqnarray*} W_1(m_x^\alpha, m_y^\alpha)  \ge  W'_1(m_x^\alpha, m_y^\alpha)  &\ge& - \left( \frac{1-\alpha}{d} -\alpha \right)  -\alpha +  \frac{1-\alpha}{d} \left(1+ 3|A| - 2|M_{1}| - |M_{2}| \right)\\
&=& - \left( \frac{1-\alpha}{d} -\alpha \right)  -\alpha +  \frac{1-\alpha}{d} \left(1+ |M_{1}| +2 |M_{2}| \right). \end{eqnarray*}

Since $\pi_T$ obtained by Algorithm \ref{alg:regular} gives the lower bound of  $W_1(m_x^\alpha, m_y^\alpha)$,  $\pi_T$ is optimal.

\medskip

\textbf{Case 2--3: $M_2 = 
M_3 = \emptyset$.}

In this case, one can transport every mass by pair of vertices of distance $1$ by Algorithm \ref{alg:regular}, and hence $\pi_T$ is optimal. Note that 
$$ W_1(m_x^\alpha, m_y^\alpha)  =  -\alpha +  \frac{1-\alpha}{d} \left(1+ 3|A| - 2|M_{1}| - |M_{2}| \right)=  -\alpha +  \frac{1-\alpha}{d} \left(1+ |M_{1}| \right). $$
\end{proof}

\begin{proof}[Proof of Theorem \ref{thm:main}.] {By the proof of Theorem \ref{thm:regularalg}, Theorem \ref{thm:main} immediately follows.}
\end{proof}

Among the diverse classes of regular graphs, Cayley graphs are highly representative. Let $\mathcal{G}$ be a finite group and let $\mathcal{S}\subseteq\mathcal{G}$ be an inverse-closed generating set that does not contain the identity element. 
The \textit{Cayley graph} $\mathrm{Cay}(\mathcal{G},\mathcal{S})$ is the graph with vertex set $\mathcal{G}$ in which two vertices $g,h\in\mathcal{G}$ are adjacent if and only if $h=gs$ for some $s\in\mathcal{S}$. 
By definition, the Cayley graph $\mathrm{Cay}(\mathcal{G}, \mathcal{S})$ is a simple graph.

\begin{example}\label{ex:power_cycles}
Suppose that $G = \operatorname{Cay}(\mathbb{Z}_n, \mathcal{S})$, where $\mathcal{S} = \{\pm 1, \pm 2, \dots, \pm k\}$ for $1 \le k < \lceil n/2 \rceil$.
For simplicity, label the vertices of $G$ by $\{0, 1, \dots, n-1\}$ in order.
Let $\ell$ denote the difference between the distinct indices $i$ and $j$ modulo $n$, where $1 \le \ell \le k$.
Then the $\alpha$-Ricci curvature is given as \Cref{tab:power_cycles}.
\begin{table}[h!]
\centering
\begin{tabular}{c|c|c}
\hline
$\frac{\kappa^\alpha(i,j)}{1-\alpha}$ & $0 \le \alpha \le \frac{1}{2k+1}$ & $\frac{1}{2k+1} \le \alpha \le 1$ \\ \hline
$n \le 2k+\ell$ & $\frac{4k-n+2}{2k}$ & $\frac{4k-n+2}{2k}$ \\ \hline
$2k+\ell < n \le 3k$ & $\frac{2k-\ell+1}{2k}$ & $\frac{2k-\ell+1}{2k}$ \\ \hline
$3k < n \le 3k+\ell$ & $\frac{1}{1-\alpha}-\frac{n+\ell-3k-1}{2k}$ & $\frac{5k-n-\ell+2}{2k}$ \\ \hline
$3k+\ell < n \le 4k+\ell$ & $\frac{1}{1-\alpha}-\frac{\ell}{k}$ & $\frac{2k-2\ell+1}{2k}$ \\ \hline
$n > 4k+\ell$ & $\frac{k-\ell}{k}$ & $\frac{k-\ell}{k}$ \\ \hline
\end{tabular}
\caption{Values of $\frac{\kappa^\alpha(i,j)}{1-\alpha}$}
\label{tab:power_cycles}
\end{table}

In particular, $\kappa^{\mathrm{LLY}}(i,j)$ is strictly positive for every edge $\{i,j\}$, except when $k = \ell$ and $n > 5k$ where it vanishes. 
This reflects the general tendency for the Lin--Lu--Yau Ricci curvature to increase as the local structure around an edge becomes more dense and symmetric.
\end{example}

{We define the \textit{diameter} $\text{diam}(G)$ of a graph $G$ by $\sup_{x,y \in V} \operatorname{dist}_G(x,y)$.}

\begin{corollary}\label{coro:diam}
Let $G=(V,E)$ be a {locally finite  graph} and $\tau:= \min\{\operatorname{diam}(G),3\}$. 
Suppose that $\{x,y \} \in E$, $\deg_G(x)=\deg_G(y)=d$ and $|N(x) \cap N(y)| = \delta$.
The $\alpha$-Ricci curvature is given by
\begin{equation*}
    \dfrac{\kappa^{\alpha}(x,y)}{1-\alpha} = 
    \begin{cases}
        \dfrac{1+2\alpha-\tau}{1-\alpha}+\dfrac{\delta}{d}+(\tau-1)\dfrac{\delta+|M_1|+2}{d} + \left\lfloor\dfrac{\tau}{3}\right\rfloor \dfrac{|M_2|}{d}& \text{ if } 0 \le \alpha \le \frac{1}{d+1}\\[10pt]
        \dfrac{\delta+2}{d}+(\tau-1)\dfrac{\delta-d+|M_1|+1}{d} + \left\lfloor\dfrac{\tau}{3}\right\rfloor \dfrac{|M_2|}{d}& \text{ if } \frac{1}{d+1} \leq \alpha \leq 1.
    \end{cases}    
\end{equation*} 
\end{corollary}
\begin{proof}
This follows immediately from Theorem \ref{thm:main}.
\end{proof}

\begin{remark}
Suppose $\operatorname{diam}(G) = 2$. Let $\nu(x,y)$ denote the size of a maximum matching between $G[N_G(x) \setminus N_G[y]]$ and $G[N_G(y) \setminus N_G[x]]$. In this case, by Corollary \ref{coro:diam} we have 
    \begin{equation}
        \kappa^{\textup{LLY}}(x,y) = \frac{\delta - d + \nu(x,y) + 3}{d}.
    \end{equation}
In particular, for a strongly regular graph with fixed parameters, $\kappa^{\textup{LLY}}(x,y)$ is determined primarily by the matching number $\nu(x,y)$. While both the $4 \times 4$ Rook's graph and the Shrikhande graph are $(16, 6, 2, 2)$-strongly regular graphs, they exhibit different curvatures ($\frac{2}{3}$ and $\frac{1}{3}$, respectively) due to differences in their matching numbers $\nu(x, y)$ (see \cite{BCDDFP20}).
\end{remark}

\section{A transportation plan on graphs}\label{sec:algorithm}
Let $G=(V,E)$ be a
graph and  
$\mu, \nu$ be probability distributions on $V$ with finite support.
We partition the vertex set $V$ based on the sign of $\mu-\nu$:
\medskip
\[
V_{+} := \{ x \in V : (\mu-\nu)(x) > 0 \},\quad V_{0} := \{ x \in V : (\mu-\nu)(x) = 0 \},\quad V_{-} := \{ x \in V : (\mu-\nu)(x) < 0 \}.
\]

We construct the transportation plan on the complete bipartite graph $K_{V_+, V_-}$ through the following iterative process.

\begin{algorithm}[Transportation plan from $\mu$ to $\nu$]\label{alg:general}
\phantom{}
\begin{enumerate}[label = Step \arabic*]
\item\label{Alg:1} Set the initial mass vector as $m^{(1,1)} := \mu-\nu$ and initialize $s := 1$.
\item\label{Alg:2} Let $V_{+}^{(s,1)} := \{x \in V : m^{(s,1)}(x) > 0\}$ and $V_{-}^{(s,1)} := \{x \in V : m^{(s,1)}(x) < 0\}$. Let $H_s$ be the subgraph of the complete bipartite graph $K_{V_{+}^{(s,1)},V_{-}^{(s,1)}}$ with the edge set 
\[
E_s := \{\{x,y\}\mid x \in V_{+}^{(s,1)}, y \in V_{-}^{(s,1)}, \operatorname{dist}_G(x,y) = s\} \subset E(K_{V_{+}^{(s,1)},V_{-}^{(s,1)}}).
\]
For each linear ordering $\sigma$ of edges in $H_s$, we define a sequence of temporary mass vectors starting with $m_{\sigma}^{(s,1)} := m^{(s,1)}$.
\item\label{Alg:3} At each step $t \geq 1$ for a fixed edge ordering $\sigma$, let $V_{+,\sigma}^{(s,t)}$ and $V_{-,\sigma}^{(s,t)}$ denote the sets of vertices with positive and negative components in $m_{\sigma}^{(s,t)}$, respectively; that is, 
\[
V_{+,\sigma}^{(s,t)} := \{x \in V : m_{\sigma}^{(s,t)}(x) > 0 \} \quad\text{and}\quad V_{-,\sigma}^{(s,t)} := \{x \in V : m_{\sigma}^{(s,t)}(x) < 0 \}.
\]
If $V_{+,\sigma}^{(s,t)}$ or $V_{-, \sigma}^{(s,t)}$ is empty, then we terminate
the process for the ordering $\sigma$ and proceed to \ref{Alg:6}. If all edges in $H_s$ have been processed according to the ordering $\sigma$, we record the total transported mass for this $\sigma$ and proceed to Step 6.
\item\label{Alg:4} Otherwise, select the next edge $e = \{x, y\} \in E_s$ from the chosen ordering $\sigma$, where $x \in V_{+,\sigma}^{(s,t)}$ and $y \in V_{-,\sigma}^{(s,t)}$. We update the temporary mass vector to $m_{\sigma}^{(s,t+1)}$ as follows:
\begin{equation}\label{eq:transport description}
m_{\sigma}^{(s,t+1)}(v) = 
\begin{cases}
m_{\sigma}^{(s,t)}(x) - \min\left(|m_{\sigma}^{(s,t)}(x)|, |m_{\sigma}^{(s,t)}(y)|\right) & \text{ if } v = x, \\ 
m_{\sigma}^{(s,t)}(y) + \min\left(|m_{\sigma}^{(s,t)}(x)|, |m_{\sigma}^{(s,t)}(y)|\right) & \text{ if } v = y, \\ 
m_{\sigma}^{(s,t)}(v) & \text{ otherwise. }
\end{cases}
\end{equation}
\item\label{Alg:5} Increment $t$ by $1$ and repeat from Step 3 for the current ordering $\sigma$.
\item\label{Alg:6} Among all linear orderings $\sigma$ of edges in $H_s$, choose one that maximizes the total transported mass. We then set $m^{(s+1,1)}$ to be the final mass vector obtained under this optimal ordering, increment $s$ by $1$, and return to \ref{Alg:2}.
\end{enumerate}   
\end{algorithm}

Since $|V_{+}\sqcup V_{-}|<\infty$ and $\mu,\nu$ are probability distributions, Algorithm~\ref{alg:general} terminates after finitely many steps and, by \eqref{eq:transport description}, the algorithm yields a valid transportation plan.

\begin{remark}
The total cost may depend on the processing order of the edges within each $H_s$.
However, we expect that the minimal total transportation cost can be attained by considering all orderings that maximize the transported mass at each step, although a general proof remains elusive.
We show that Algorithm \ref{alg:general} successfully yields an optimal transportation plan not only for the regular graph cases covered by Theorem \ref{thm:regularalg}, but also across all subsequent results in the following sections where explicit computational results are obtained.
\end{remark}

Nonetheless, even when the algorithm does not immediately guarantee the optimal transportation cost for a general graph, we show that one can construct a function on $V_+ \sqcup V_-$ to compute this cost directly. 
Verifying the $1$-Lipschitz property of this function %on the entire graph $G$ 
then certifies the optimality of the plan.

\begin{lemma}\label{lem:acyclic}
    Let $G'=(V_{+} \sqcup V_{-},E')$ be the subgraph of $K_{V_+,V_-}$ such that $E'$ is the collection of edges in $K_{V_+,V_-}$ used for the transportation plan produced by Algorithm \ref{alg:general}. 
   Then $G'$ is acyclic.
\end{lemma}

\begin{proof}
By construction, $G'$ is bipartite with the bipartition $V' = V_+ \sqcup V_-$. Assume, for the sake of contradiction, that $G'$ contains a cycle $C$. Since $G'$ is bipartite, $C$ must alternate between vertices in $V_+$ and vertices in $V_-$. Label these vertices such that $V(C) = \{x_1,\dots,x_m\} \cup \{y_1,\dots,y_m\}$ for $x_i \in V_+$, $y_i \in V_-$, and
\[
E(C) = \left\{\{x_i,y_i\} \mid 1 \le i \le m\right\} \cup \left\{ \{y_i,x_{i+1}\} \mid 1 \le i \le m-1\right\} \cup \left\{\{y_m,x_1\}\right\}.
\]
For simplicity, we omit the transport indices $(s,t)$ (i.e., we denote the mass vector at any step simply by $m$) and regard the graph $G'$ as a subgraph of $K_{V_+,V_-}$ without explicit reference to the graph distance $\operatorname{dist}_G$.

Consider the first edge $\{x_1,y_1\}$ of $C$ that is processed by the algorithm. In the updating process, the mass transferred along $\{x_1,y_1\}$ is $\min\left(|m(x_1)|, |m(y_1)|\right)$. After this transfer, at least one of $m(x_1)$ or $m(y_1)$ is reduced to zero. 

If $m(y_1)=0$, then the edge $\{y_1,x_2\}$ of $C$ can no longer be chosen in the plan. Similarly, if $m(x_1)=0$, then the edge $\{y_m,x_1\}$ can no longer be used. 
In either case, at least one edge of $C$ cannot be used in the transportation, contradicting the assumption that all edges of $C$ belong to $E'$. 
Hence, $G'$ contains no cycles.
\end{proof}

\begin{remark}
Since the graph $G'$ in Lemma \ref{lem:acyclic} is a forest, the flow problem on $G'$ satisfies 
$$
\max_{\phi \in \Phi(G')} \sum_{e \in E'} \phi(e)
=
\min_{S \in \mathcal{S}(G')} \sum_{v \in S} b(v),
$$
where $\Phi(G')$ is the set of feasible flows, $\mathcal{S}(G')$ is the family of vertex covers of $G'$, and the vertex capacity function $b:V\to\mathbb{R}_{\ge 0}$ is {defined by $b(x) = |(\mu-\nu)(x)|$} .  This follows from the fact that the incidence matrix of a bipartite graph is totally unimodular and using the strong duality theorem of linear programming (see \cite[Section 19.1]{Sch86}).
\end{remark}

\begin{theorem}\label{prop:TransportationPlan}
Let $G=(V,E)$ be a %connected 
graph and let $\mu$ and $\nu$ be probability distributions on $V$. Suppose that the difference $\mu-\nu$ has finite support. 
Let $P$ be the transportation plan obtained by Algorithm \ref{alg:general} from $\mu$ to $\nu$. Let $G'$ be the subgraph of $K=K_{V_+,V_-}$ consisting of the vertices and edges used by $P$, with a bipartition $V_{+} \sqcup V_{-}$. 
Then there exists a function $f: V_{+} \sqcup V_{-} \rightarrow \mathbb{R} $ such that 
\begin{equation}\label{pseudo lipschitz}
f(x)-f(y) = \textup{dist}_G (x,y), \quad \text{for all} \;\; \{ x, y \} \in E(G') \; \text{with} \; \;  x \in V_{+}, \; y \in V_{-}
\end{equation}
and the cost of $P$ is 
\begin{equation}\label{equ:totalcost}
     \sum_{v \in V} f(v) \cdot \left(\mu(v)-\nu(v)\right).
\end{equation}
{If the function $f:V_{+}\sqcup V_{-}\to\mathbb{R}$ is
$1$-Lipschitz on $V_{+}\sqcup V_{-}$ with respect to
$\operatorname{dist}_G$,} then $P$ is optimal. In particular,
the cost of $P$ equals $W_1(\mu,\nu)$.
\end{theorem}
\begin{proof}
By Lemma \ref{lem:acyclic}, $G'$ is acyclic.
Let $\widetilde G$ be a spanning tree of $K_{V_+,V_-}$ containing $G'$. Choose a vertex $x_1\in V_+ \sqcup V_{-}$ and assign an arbitrary real value $\eta$ to $f(x_1)$. For each vertex
$v\in V(\widetilde G)\setminus\{x_1\}$, let
$ x_1=v_0,v_1,\ldots,v_k=v $
be the unique path joining $x_1$ and $v$ in $\widetilde G$.
We define $f(v_j)$ inductively for $j=1,\cdots, k$ by
\begin{equation*}
 f(v_j) := \begin{cases}
        f(v_{j-1}) + \textup{dist}_G(v_{j-1},v_j) & \text{if} \; v_{j} \in V_{+}, \; \{ v_{j-1}, v_j \} \in E (\widetilde G) \\
        f(v_{j-1}) - \textup{dist}_G(v_{j-1},v_j) & \text{if} \; v_j \in V_{-}, \; \{ v_{j-1}, v_j \} \in E(\widetilde G)
        \end{cases}
\end{equation*}
Since $\widetilde G$ is a tree, $f$ is a well-defined function on $V(\widetilde G)$, and it satisfies
$$
f(x)-f(y) = \textup{dist}_G(x,y) \;\; \text{for} \;\; \{ x, y \} \in E(\widetilde G) \;\;
\text{with} \;\; x \in V_{+},  \; y \in V_{-}.
$$
Therefore, \eqref{pseudo lipschitz} follows since $G'$ is a subgraph of $\widetilde G$.

The remaining part is to calculate the cost of $P$. We write
\[
V_{+}=\{x_1,\ldots,x_m\}
\quad\text{and}\quad
V_{-}=\{y_1,\ldots,y_n\}.
\]
For each edge in $E(K)$ joining $x_i\in V_{+}$ and $y_j\in V_{-}$, we denote it by $e_{ij}$ and write $m_{ij}$ for the mass transported along $e_{ij}$, where $i=1,\ldots,m$ and $j=1,\ldots,n$.
By \eqref{pseudo lipschitz}, the cost of $P$ is calculated as
\begin{equation*}
\begin{aligned}
\textup{Cost}(P) = \sum_{i,j} m_{ij} (f(x_i) - f(y_j) ) &= \sum_{i} f(x_i) \bigg( \sum_{y_j \in N_{G'} (x_i) } m_{ij} \bigg) - \sum_j f(y_j) \bigg( \sum_{x_i \in N_{G'} (y_j )} m_{ij} \bigg) \\
&=\sum_{i} f(x_i) \big| (\mu- \nu) (x_i) \big| - \sum_{j} f(y_j) |(\mu - \nu) (y_j) | \\
&=\sum_{x \in V_{+}} f(x) (\mu(x)-\nu(x)) + \sum_{y \in V_{-}} f(y) (\mu(y) - \nu (y)) \\
&= \sum_{v \in V} f(v) (\mu(v)-\nu(v)).
\end{aligned}
\end{equation*} 

If the function $f:V_{+}\sqcup V_{-}\to\mathbb{R}$ in Theorem
\ref{prop:TransportationPlan} is $1$-Lipschitz on
$V_{+}\sqcup V_{-}$ with respect to $\operatorname{dist}_G$, then by \cite[Theorem 1]{Mcshane},
$f$ extends to a $1$-Lipschitz function on $V$.
Therefore, by Theorem \ref{thm:KR_duality}, the transportation cost of $P$
equals $W_1(\mu,\nu)$.
This completes the proof.
\end{proof}

\section{Ranges of the $\alpha$-curvatures of graphs}\label{sec:range}

In what follows, we focus on the $1$-Wasserstein distance between certain probability mass distributions arising in the study of Ricci curvature on graphs. 
From now on, all discussions will be based on the distribution function $m_{x}^{\alpha}$ given by \eqref{mass funcion for alpha-Ric}. 

\begin{lemma}\label{lem:Wrange}
Let $G=(V,E)$ be a {locally finite   graph}
with vertex set {$V=\{x_1,x_2,\cdots\}$. } {For an edge $\{ x_i,x_j \} \in E(G)$, let $m_i^\alpha := m_{x_i}^{\alpha}, m_j^{\alpha} := m_{x_j}^{\alpha}$, and $d_i := \operatorname{deg}_{G} (x_i), d_j := \operatorname{deg}_{G} (x_j)$. Assume that $d_i \leq d_j$. 
Then
\[
\frac{1}{2}\,\|m^\alpha_i - m^\alpha_j\|_1
\;\leq\;
W_1(m^\alpha_i,m^\alpha_j)
\;\leq\;
\frac{3}{2}\,\|m^\alpha_i - m^\alpha_j\|_1 - \Delta_{ij}(\alpha),
\]
where 
$
\| m_i^{\alpha} - m_j^\alpha \|_1 := \sum_{v \in V} |m_i^\alpha (v) - m_j^\alpha (v) |
$
and
\[
\Delta_{ij}(\alpha)
:=
\begin{cases}
2\alpha - \tfrac{1-\alpha}{d_i} - \tfrac{1-\alpha}{d_j}
& \text{if }  \frac{1}{d_i+1} \leq \alpha \leq 1, \\

\frac{2(1-\alpha)}{d_i}-\alpha-\frac{1-\alpha}{d_j}
& \text{if } \frac{1}{d_j+1}\leq \alpha \leq \frac{1}{d_i+1}, \\

2(1-2\alpha)
& \text{if } 0 \leq \alpha \leq \frac{1}{d_j+1}
\text{ and } d_i=1, \\

2\left(\frac{1-\alpha}{d_j}+\frac{1-\alpha}{d_i}-2\alpha\right)
& \text{if } 0 \leq \alpha \leq \frac{1}{d_j+1}
\text{ and } d_i \ge 2.
\end{cases}
\]
Moreover,
\begin{itemize}
    \item Equality in the lower bound holds if and only if 
    there exists a transportation plan such that any mass in $G$ is transported along paths of length 1 in $G$.
    \item Equality in the upper bound holds if and only if the edge $\{x_i,x_j\}$ does not lie in any induced cycle of length $3$, $4$, or $5$ in $G$.
\end{itemize}
}
\end{lemma}

\begin{proof}
Let $P$ be the transportation plan for $m_i^{\alpha}$ and $m_j^{\alpha}$ given by Algorithm \ref{alg:general}. 
If the mass is transported directly along the edges of $G$ (i.e., using only paths of length $1$ in $G$), the resulting transportation cost is precisely $\tfrac{1}{2}\,\|m^\alpha_i - m^\alpha_j\|_1$. 
Since any valid transportation plan on $G$ cannot have a cost lower than this value, we immediately obtain the lower bound
\[
\tfrac{1}{2}\,\|m^\alpha_i - m^\alpha_j\|_1 \;\le\; W_1(m^\alpha_i,m^\alpha_j).
\]
To verify the equality case, suppose that the transportation plan $P$ utilizes only paths of length $1$ in $G$. We then define {a function $f: V \to \mathbb{R}$ by
\[
f(v) = \begin{cases}
    1 & \text{if } v \in V_+, \\
    0 & \text{otherwise.}
\end{cases}
\]
Since $f$ is $1$-Lipschitz on $V$,} Theorem \ref{thm:KR_duality} implies that the lower bound is tight; that is, 
\[
W_1(m_i^\alpha, m_j^\alpha) = \tfrac{1}{2}\,\|m^\alpha_i - m^\alpha_j\|_1.
\]

Furthermore, if any positive amount of mass is transported along a path of length at least two, the total transportation cost must be strictly greater than $\tfrac{1}{2}\,\|m^\alpha_i - m^\alpha_j\|_1$.

To derive the upper bounds, we analyze the worst-case scenario for the transportation cost.
More precisely, we assume that every unit of mass transported between $N_G(x_i)\setminus\{x_j\}$ and $N_G(x_j)\setminus\{x_i\}$ must travel along a path of length three (that is, the transport is forced to pass through $x_i$ and $x_j$),
and moreover that $N_G(x_i)\cap N_G(x_j)=\emptyset$.

Let $\tau_1$ denote the total mass sent along edges of length $1$, $\tau_2$ the total mass sent along edges of length $2$, and $\tau_3$ the total mass sent along edges of length $3$.

We distinguish the following cases according to the value of $\alpha$ and the degree $d_i$:
\[
\alpha \in 
\Bigl[\tfrac{1}{d_i+1},\,1\Bigr],\quad
\Bigl[\tfrac{1}{d_j+1},\,\tfrac{1}{d_i+1}\Bigr],\quad
\text{or}\quad
\alpha \le \tfrac{1}{d_j+1},
\]
where in the last case we further distinguish the subcases $d_i=1$ and $d_i\ge 2$.

\medskip
\noindent
\textbf{Case I: $\boldsymbol{\tfrac{1}{d_i+1}\le \alpha \le 1}$.}
For an edge $\{ x_i,x_j \} \in E(G)$, it is optimal to transport at least $\mu_1 \;\ge\; \alpha - \frac{1-\alpha}{d_i}$ units of mass directly from $x_i$ to $x_j$.
If $x_i$ has no neighbor in $N_G(x_j)$, then an additional amount $\mu_2 \;\ge\; \frac{1-\alpha}{d_i} - \frac{1-\alpha}{d_j}$ must be transported from $x_i$ to $N_G(x_j)$ along paths of length at least two.

Define
\[
\mu_1' := \mu_1 - \Bigl(\alpha - \tfrac{1-\alpha}{d_i}\Bigr)\ge 0,
\qquad
\mu_2' := \mu_2 - \Bigl(\tfrac{1-\alpha}{d_i} - \tfrac{1-\alpha}{d_j}\Bigr)\ge 0.
\]
Then
\begin{align*}
\mathrm{Cost}(P) = \mu_1+2\mu_2+3\mu_3
&= \mu_1' + 2\mu_2' + 3\mu_3
   + \Bigl(\alpha - \tfrac{1-\alpha}{d_i}\Bigr)
   + 2\Bigl(\tfrac{1-\alpha}{d_i} - \tfrac{1-\alpha}{d_j}\Bigr) \\
&\le 3(\mu_1+\mu_2+\mu_3)
   - 2\Bigl(\alpha - \tfrac{1-\alpha}{d_i}\Bigr)
   - \Bigl(\tfrac{1-\alpha}{d_i} - \tfrac{1-\alpha}{d_j}\Bigr).
\end{align*}
Since $\mu_1+\mu_2+\mu_3=\tfrac{1}{2}\|m^\alpha_i-m^\alpha_j\|_1$, we obtain
\[
W_1(m^\alpha_i,m^\alpha_j)
\;\le\; \tfrac{3}{2}\,\|m^\alpha_i - m^\alpha_j\|_1
- \Bigl(2\alpha - \tfrac{1-\alpha}{d_i} - \tfrac{1-\alpha}{d_j}\Bigr).
\]

\medskip
\noindent
\textbf{Case II: $\boldsymbol{\tfrac{1}{d_j+1}\le \alpha \le \tfrac{1}{d_i+1}}$.}
Here $\mu_1 \ge \tfrac{1-\alpha}{d_i}-\alpha$ and $\mu_2 \ge \alpha-\tfrac{1-\alpha}{d_j}$.
Setting $\mu_1' := \mu_1 - \Bigl(\tfrac{1-\alpha}{d_i}-\alpha\Bigr)$ and $\mu_2' := \mu_2 - \Bigl(\alpha-\tfrac{1-\alpha}{d_j}\Bigr)$, and repeating the argument of Case~I yields
\[
W_1(m^\alpha_i,m^\alpha_j)
\;\le\; \tfrac{3}{2}\,\|m^\alpha_i - m^\alpha_j\|_1
- \Bigl(\tfrac{2(1-\alpha)}{d_i}-\alpha-\tfrac{1-\alpha}{d_j}\Bigr).
\]
\medskip
\noindent
\textbf{Case III: $\boldsymbol{\alpha \le \tfrac{1}{d_j+1}}$.}
\smallskip
\noindent
\emph{Subcase III(a): $d_i=1$.}
In this case $\mu_1 \ge 1-2\alpha$ and $\mu_2\ge 0$.
Setting $\mu_1':=\mu_1-(1-2\alpha)$ and $\mu_2':=\mu_2$, we obtain
\[
W_1(m^\alpha_i,m^\alpha_j)
\;\le\; \tfrac{3}{2}\,\|m^\alpha_i - m^\alpha_j\|_1-2(1-2\alpha).
\]
\smallskip
\noindent
\emph{Subcase III(b): $d_i\ge 2$.}
Here $\mu_1\ge \tfrac{1-\alpha}{d_j}+\tfrac{1-\alpha}{d_i}-2\alpha$ and $\mu_2\ge 0$.
Setting
$\mu_1' := \mu_1 - \Bigl(\tfrac{1-\alpha}{d_j}+\tfrac{1-\alpha}{d_i}-2\alpha\Bigr)$ and $\mu_2':=\mu_2$,
we conclude that
\[
W_1(m^\alpha_i,m^\alpha_j)
\;\le\; \tfrac{3}{2}\,\|m^\alpha_i - m^\alpha_j\|_1
-2\Bigl(\tfrac{1-\alpha}{d_j}+\tfrac{1-\alpha}{d_i}-2\alpha\Bigr).
\]
\medskip
In all cases, the edge $\{x_i,x_j\}$ does not lie in any cycle $C_3$, $C_4$, or $C_5$ if and only if the worst-case configuration occurs, namely $\mu_1'=\mu_2'=0$.
In this situation, equality is attained in the corresponding upper bound since we can define a function $f: V_+\sqcup V_- \to \mathbb{R}$ by
$$
f(v) = \begin{cases}
    2 & \text{if } v\in N_G(x_i)\setminus N_G[x_j], \\
    1 & \text{if } v = x_i, \\
    0 & \text{if } v = x_j\\
    -1 & \text{if } v\in N_G(x_j)\setminus N_G[x_i]\\
\end{cases}
$$
for Cases I, II, and III(b), and by
$$
f(v) = \begin{cases}
    1 & \text{if } v = x_j, \\
    0 & \text{otherwise}\\
\end{cases}
$$
for Case III(a). By Theorem \ref{prop:TransportationPlan}, it follows that the upper bound is tight. 
\end{proof}

\begin{remark}
We note that Lemma \ref{lem:Wrange} is applicable to a more general setting.  Let $\{ x_i, x_j \} \in E(G)$, and let $\mu$ and $\nu$ be two probability distributions on $V$ associated with $x_i$ and $x_j$, respectively. 
Suppose $\operatorname{supp}(\mu - \nu) \subseteq N_G[x_i] \cup N_G[x_j]$. 
Then, we have
\[
\frac{1}{2}\,\|\mu - \nu \|_1
\;\leq\;
W_1(\mu, \nu)
\;\leq\;
\frac{3}{2}\,\|\mu - \nu \|_1 - \Delta,
\]
where $\Delta \geq 0$ is a correction term depending on the local mass distribution around $x_i$ and $x_j$.
\end{remark}

\begin{theorem}\label{thm:CurvatureRange}
Let $G=(V,E)$ be a {locally finite  graph} and let $\{ x_i,x_j \} \in E$. 
Assume $d_i \leq d_j$, where $d_i := \deg_G(x_i)$ and $d_j := \deg_G(x_j)$, and define $\delta := |N_G(x_i) \cap N_G(x_j)|$.
Then
\begin{equation}\label{equ:CurvatureRangeLowerBdd}
    -2+\frac{2}{d_i} + \frac{2+3\delta}{d_j} \;\leq\; \frac{\kappa^{\alpha}(x_i,x_j)}{1-\alpha} \quad\text{ if } 
    \begin{cases}
        \frac{1}{d_j+1} \leq \alpha \leq 1 \text{ or } \\
        0 \leq \alpha \leq \frac{1}{d_j+1} \text{ and } d_i \geq 2,
    \end{cases}
\end{equation}
\begin{equation}\label{equ:CurvatureRangeUpperBdd}
   \frac{\kappa^{\alpha}(x_i,x_j)}{1-\alpha} \;\leq\; 
   \begin{cases}
       \frac{1+\delta}{d_i}+\frac{1}{d_j} & \text{ if } \frac{1}{d_i+1} \leq \alpha \leq 1\\
       \frac{\alpha}{1-\alpha}+\frac{1+\delta}{d_j} & \text{ if } \frac{1}{d_j+1} \leq \alpha \leq \frac{1}{d_i+1} \\
       \frac{2\alpha}{1-\alpha} +\frac{\delta}{d_j} & \text{ if } 0 \leq \alpha \leq \frac{1}{d_j+1} \text{ and } d_i \geq 2 \text{, and}\\
   \end{cases}
\end{equation}
\begin{equation*}\label{equ:CurvatureRangeBdd}
    \kappa^{\alpha}(x_i,x_j) = 2\alpha \quad\text{ if } 0\leq \alpha \leq \frac{1}{d_j+1} \text{ and } d_i = 1.
\end{equation*}
Moreover,
\begin{itemize}
    \item Equality in \eqref{equ:CurvatureRangeLowerBdd} holds if and only if the edge $\{x_i,x_j\}$ does not lie in any induced cycle of length $3$, $4$, or $5$ in $G$.
    \item Equality in \eqref{equ:CurvatureRangeUpperBdd} holds if and only if there is a transport plan such that any mass in $G$ is transported by path of length $1$ in $G$. 
    Especially, if $d_i=d_j$, then equality in \eqref{equ:CurvatureRangeUpperBdd} holds if and only if {there exists a perfect matching between $N_G(x_i) \setminus N_G[x_j]$ and $N_G(x_j) \setminus N_G[x_i]$.}
\end{itemize}
\end{theorem}
\begin{proof}
Let $m_i^\alpha := m_{x_i}^{\alpha}$ and $m_j^{\alpha} := m_{x_j}^{\alpha}$. One can compute $ \| m_i^{\alpha} - m_j^\alpha \|_1 := \sum_{v \in V} |m_i^\alpha (v) - m_j^\alpha (v) |$ in each case:
\begin{equation}\label{equ:1norm}
\|m^\alpha_i-m^\alpha_j\|_1 
= 
\begin{cases}
    2- 2(1-\alpha)\left(\frac{1}{d_i}+\frac{\delta+1}{d_j}\right)& \text{ if } \frac{1}{d_i+1} \leq \alpha \leq 1\\
    2(1-\alpha)\left(1+\frac{-\delta+1}{d_j}\right)& \text{ if } \frac{1}{d_j+1} \leq \alpha \leq \frac{1}{d_i+1}\\
    2-4\alpha-2(1-\alpha)\frac{\delta}{d_j}& \text{ if } 0 \leq \alpha \leq \frac{1}{d_j+1}.\\
\end{cases}
\end{equation}
Applying Lemma \ref{lem:Wrange} together with \eqref{equ:1norm} yields the desired bounds for $\kappa^{\alpha}(x_i,x_j)$.
\end{proof}

We emphasize that if there is a transportation plan such that any mass in $G$ is transported along paths of length $1$ in $G$,
then the Lin--Lu--Yau Ricci curvature $\kappa^{\mathrm{LLY}}(v_i,v_j)$ is always positive. 

\begin{corollary}\label{alg:positive ricci for transport-max}
For $\{x_i, x_j \} \in E$, let $m_i^\alpha := m_{x_i}^{\alpha}$, $m_j^{\alpha} := m_{x_j}^{\alpha}$, and 
$
\|m_i^\alpha - m_j^\alpha \|_1 := \sum_{v \in V} |m_i^\alpha (v) - m_j^\alpha (v) |.
$
If $W_1(m^\alpha_i,m^\alpha_j) = \tfrac{1}{2}\,\|m^\alpha_i - m^\alpha_j\|_1$, then $\kappa^{\mathrm{LLY}}(x_i,x_j) > 0.$
\end{corollary}
\begin{proof}
One has $\|m^\alpha_i - m^\alpha_j\|_1 \le \|m^\alpha_i\|_1 + \|m^\alpha_j\|_1 = 2$ with equality if and only if
\[
|m^\alpha_i(v)-m^\alpha_j(v)| 
= |m^\alpha_i(v)| + |m^\alpha_j(v)|
\quad \text{for all} ~ v \in V.
\]
However, the equality condition always fails for $v=x_i$ and $v=x_j$, and hence $ \|m^\alpha_i - m^\alpha_j\|_1 < 2$.
Therefore,
\[
\kappa^{\mathrm{LLY}}(x_i,x_j)
= \lim_{\alpha\to1^-} \frac{1-\tfrac{1}{2}\|m^\alpha_i-m^\alpha_j\|_1}{1-\alpha}
> 0.
\]  
\end{proof}
In addition, we provide a condition on the number of common neighbors that guarantees non-negative Lin--Lu--Yau Ricci curvature.

\begin{proof}[Proof of Theorem \ref{thm:LLY_nonnegative}]
We construct a feasible transportation plan $\pi$ from $m^\alpha_x$ to $m^\alpha_y$ as follows:
    \begin{itemize}
        \item Transport a mass of $\alpha - \frac{1-\alpha}{d_x}$ from $x$ to $y$ along the edge $\{x, y\}$.          
        \item Transport a mass of $\frac{1-\alpha}{d_x}-\frac{1-\alpha}{d_y}$ from each vertex in $N_G[x]\cap N_G(y)$ to $N_G(y)\setminus N_G[x]$ via paths of length $2$.
        \item Transport a mass of $\frac{1-\alpha}{d_x}$ from each vertex in $N_G(x)\setminus N_G[y]$ to $N_G(y)\setminus N_G[x]$ via paths of length $3$.
    \end{itemize}
    The total cost of this transportation plan is given by
$$
    \alpha + (1-\alpha)\left(3 - \frac{\delta+2}{d_x} - \frac{2\delta+2}{d_y}\right).
$$
Under the assumptions that $E(\Delta, B) = \emptyset$ and $\operatorname{dist}_G(u,w) = 3$ for all $u \in A$ and $w \in B$, we verify the optimality of this plan by defining a function $f: V_+ \sqcup V_- \to \mathbb{R}$ as follows:
    $$
    f(v) = \begin{cases}
        2 & \text{if } v \in N_G(x)\setminus N_G[y], \\
        1 & \text{if } v \in N_G[x]\cap N_G(y), \\
        0 & \text{if } v = y, \\
        -1 & \text{if } v \in N_G(y)\setminus N_G[x].
    \end{cases}
    $$
Since $f$ is a $1$-Lipschitz function on $V_{+} \sqcup V_{-}$ with respect to $\operatorname{dist}_G$, 
% and the cost is $\sum_{v\in V} f(v)\cdot(m^\alpha_x(v) - m^\alpha_y(v))$. 
by \cite[Theorem 1]{Mcshane}, it follows that $W_1(m^\alpha_x,m^\alpha_y) = \operatorname{Cost}(\pi)$ in this case.

 A direct computation shows that $\delta \geq \frac{2d_x d_y - 2d_x - 2d_y}{2d_x+d_y}$ implies $\kappa^{\textup{LLY}}(x,y) \geq 0$. 
Furthermore, under the assumption $\delta \geq \frac{2d_x d_y - 2d_x - 2d_y}{2d_x+d_y}$, we can establish the exact necessary and sufficient conditions for $\kappa^{\textup{LLY}}(x,y) = 0$ by constructing that an alternative transportation plan $\pi'$ with a strictly lower cost can be constructed whenever any of these conditions fail.
\end{proof}

\section{Ricci Curvature of Graph Joins}\label{sec:join}
In this section, we apply our main results to graphs with certain prescribed structures. 
Let $G_1=(V_1,E_1)$ and $G_2=(V_2,E_2)$ be disjoint graphs. 
The \textit{graph join} $G_1\vee G_2$ is the graph with vertex set $V_1\cup V_2$ and edge set 
$E_1\cup E_2\cup \{ \{u,v \} \mid u\in V_1,\ v\in V_2 \}.$
For $l$ pairwise disjoint graphs $G_1, \dots, G_l$, this operation naturally extends to the join $G_1\vee\cdots\vee G_l$ of multiple graphs.

The join operation naturally yields a graph of diameter at most two. 
This property facilitates an exact determination of the $W_1$-distance and the subsequent analysis of the $\alpha$-Ricci curvature.

\begin{proposition}\label{prop:graphjoins}
Let $G := G_1 \vee \dots \vee G_l$ be the join of finite graphs  $G_1, \dots, G_l$. 
For each $i \in \{1, 2, \dots, l\}$, let $n_i := |V(G_i)|$ and $n := |V(G)| = n_1 + n_2 + \dots + n_l$. 
For distinct indices $i,j$, let $x_i \in V(G_i)$ and $x_j \in V(G_j)$ satisfy $d_i := \deg_G(x_i) \le d_j := \deg_G(x_j)$, and denote the local degree by $d_i' := \deg_{G_i}(x_i)$.
Then
\[
\frac{\kappa^{\alpha}(x_i, x_j)}{1-\alpha}
\begin{cases}
    \geq 1 - \frac{n - d_i - d'_i - 2}{d_j} - \frac{d'_i}{d_i} & \text{ if } \frac{1}{d_j+1} \leq \alpha \leq 1 \\
    \geq 1- \frac{d'_i-1}{d_i} - \frac{n-d_i-d'_i-1}{d_j} & \text{ if } \max\{0, \frac{d_i-d'_i(d_j-d_i)}{d_id_j+d_i-d'_i(d_j-d_i)}\} \leq \alpha \leq \frac{1}{d_j+1}\\
    = \frac{1+\alpha}{1-\alpha}- \frac{n-d_i}{d_j}& \text{ if } 0 \leq \alpha \leq \max\{0, \frac{d_i-d'_i(d_j-d_i)}{d_id_j+d_i-d'_i(d_j-d_i)}\}\\
\end{cases}
\]
Equality holds if and only if the induced subgraphs $G[N_{G_{i}}[x_{i}]]$ and $G[V_{i} \setminus N_{G_{i}}[x_{i}]]$ are disjoint.
In particular, $\kappa^{\mathrm{LLY}}(x_{i}, x_{j}) > 0$ whenever $d_i \leq d_j \leq 2d_i$ or $d'_i < \frac{d_{i}+d_{j}-n+2}{d_j/d_i-1}$.
\end{proposition}

\begin{proof}
The proof follows the methodology based on Algorithm \ref{alg:general}. 
Partition the vertex set as
\[
V = \{x_{i}\} \sqcup \{x_{j}\} \sqcup N_{G_{i}}(x_{i}) \sqcup (V_{i} \setminus N_{G_{i}}[x_{i}]) \sqcup N_{G_{j}}(x_{j}) \sqcup (V_{j} \setminus N_{G_{j}}[x_{j}]) \sqcup (V \setminus (V_{i} \cup V_{j})).
\]
Let $m_i^{\alpha} := m_{x_i}^{\alpha}$ and $m_j^{\alpha} := m_{x_j}^{\alpha}$.
Then $m = m_{i}^{\alpha} - m_{j}^{\alpha}$ is given by
\begin{equation*}
\begin{aligned}
m(v) = \begin{cases}
 \alpha-\tfrac{1-\alpha}{d_{j}} & \text{if} \; v  = x_i    \\
    \tfrac{1-\alpha}{d_{i}}-\alpha & \text{if} \; v = x_j  \\
    \tfrac{1-\alpha}{d_{i}}-\tfrac{1-\alpha}{d_{j}} & \text{if} \; v \in N_{G_i} (x_i) \sqcup N_{G_j} (x_j) \sqcup (V \setminus (V_i \sqcup V_j ) \\
    -\tfrac{1-\alpha}{d_{j}}, & \text{if} \; v \in V_i \setminus N_{G_i}[x_i] \\
    \tfrac{1-\alpha}{d_{i}} & \text{if} \; v \in V_j \setminus N_{G_j} [x_j]. 
          \end{cases}
\end{aligned}
\end{equation*}

There are the four cases according to the value of $\alpha$: 
$\frac{1}{d_i+1} \leq \alpha \leq 1$, $\frac{1}{d_j+1} \leq \alpha \leq \frac{1}{d_i+1}$, $\max\!\left\{0, \frac{d_i - d'_i(d_j-d_i)}{d_i d_j + d_i - d'_i(d_j-d_i)}\right\} \leq \alpha \leq \frac{1}{d_j+1}$, and $0 \leq \alpha \leq \max\!\left\{0, \frac{d_i - d'_i(d_j-d_i)}{d_i d_j + d_i - d'_i(d_j-d_i)}\right\}$.

We present the computation for the first case, namely 
$\frac{1}{d_i+1} \leq \alpha \leq 1$; 
the remaining cases follow by the same argument.
Note that the values of $W_1(m_i^\alpha,m_j^\alpha)$ in the first two cases coincide, and hence these two cases can be merged.

To obtain the lower bound on {$\kappa^\alpha(x_i,x_j)$}, we assume there is no edge between $N_{G_i}(x_i)$ and $V_i \setminus N_{G_i}[x_i]$. {When $s = 1$ in Algorithm \ref{alg:general}}, transport an amount $\alpha - \frac{1-\alpha}{d_j}$ from $x_i$ to $x_j$, and send $\frac{1-\alpha}{d_i}$ from each vertex in $V_j \setminus N_{G_j}[x_j]$, as well as $\frac{1-\alpha}{d_i}-\frac{1-\alpha}{d_j}$ from each vertex in $N_{G_j}(x_j)$ and $V\setminus (V_i\cup V_j)$, to vertices in $V_i \setminus N_{G_i}[x_i]$.
When $s = 2$ in Algorithm \ref{alg:general}, transport the remaining mass $\frac{1-\alpha}{d_{i}} - \frac{1-\alpha}{d_{j}}$ from each vertex in $N_{G_{i}}[x_i]$ to the vertices in  $V_{i} \setminus N_{G_{i}}[x_{i}]$.

This transportation plan yields the upper bound
$$
W_{1}(m_{i}^{\alpha}, m_{j}^{\alpha}) \leq \alpha + \left(\frac{n - d_{i} - d'_{i} - 2}{d_j} + \frac{d'_{i}}{d_{i}}\right)(1-\alpha).
$$
This bound is tight if and only if the induced subgraphs $G[N_{G_{i}}[x_{i}]]$ and $G[V_{i} \setminus N_{G_{i}}[x_{i}]]$ are disjoint.
Indeed, when this holds, optimality is certified by the $1$-Lipschitz function $f:V_-\sqcup V_+ \to \mathbb{R}$ defined as
$$
f(v) = \begin{cases} 
    1 & \text{if } v \in N_{G_{i}}[x_{i}], \\ 
    -1 & \text{if } v \in V_{i} \setminus N_{G_{i}}[x_{i}], \\ 
    0 & \text{otherwise.} 
\end{cases}
$$
Conversely, the existence of an edge between $N_{G_i}(x_i)$ and $V_{i} \setminus N_{G_{i}}[x_{i}]$ enables mass transport along a strictly shorter path, and therefore strictly reduces the total transportation cost.
\end{proof}

\begin{corollary}\label{cor:specialjoingraphs}
With the same assumptions as in Proposition \ref{prop:graphjoins}, the following hold:
\begin{enumerate}
    \item\label{cor:specialjoingraphs1} If $d_i = d_j =d$, then 
    \[
    \frac{\kappa^\alpha(x_i,x_j)}{1-\alpha} =
    \begin{cases}
        2- \frac{n-2}{d} & \text{ if } \frac{1}{d+1}\leq \alpha \leq 1 \\
        \frac{2}{1-\alpha} - \frac{n}{d} & \text{ if } 0 \leq \alpha \leq \frac{1}{d+1}.
    \end{cases}
    \]
    \item\label{cor:specialjoingraphs2} For $n_1, n_2, \dots, n_\ell\in \mathbb{N}$ let $K_{n_{1},n_{2},\dots,n_{\ell}}$ be the complete $\ell$-partite graph.
    For any $x_{i} \in V_{i}$ and $x_{j} \in V_{j}$ with $i \neq j$ and $n_i \geq n_j$,
    \[
    \frac{\kappa^{\alpha}(x_{i},x_{j})}{1-\alpha} = 
    \begin{cases}
        1-\frac{n_i-2}{n-n_j} & \text{ if } \frac{1}{d_j+1} \leq \alpha \leq 1\\
        \frac{1+\alpha}{1-\alpha} -\frac{n_i}{n-n_j} & \text{ if } 0 \leq \alpha \leq \frac{1}{d_j+1}.
    \end{cases}
    \]
    \item\label{cor:specialjoingraphs3} If $N_{G_{i}}[x_{i}] = V(G_{i})$ (i.e., $d'_{i} = n_{i}-1$), then
    \[
    \frac{\kappa^{\alpha}(x_{i}, x_{j})}{1-\alpha} = 
    \begin{cases}
        \frac{n}{n-1} & \text{ if } \frac{1}{n}\leq \alpha \leq 1 \\
        \frac{2}{1-\alpha} - \frac{n}{n-1} & \text{ if } 0 \leq \alpha \leq \frac{1}{n}\\
    \end{cases}
    \]
\end{enumerate}
In particular, for every such edge $\{x_{i},x_{j}\}$, the Lin-Lu-Yau Ricci curvature is strictly positive.
\end{corollary}

\begin{example} \label{ex:wheel graph}
Let $C_n$ be a cycle graph on $n$ vertices and $K_1$ be the graph consisting of a single vertex.
For $\{x_1,x_2\}\in E(C_n \vee K_1)$ with $x_1 \in V(C_n)$ and $x_2 \in V(K_1)$, $\kappa^{\mathrm{LLY}}(x_1,x_2)$ is negative if $n\geq 13$.
On the other hand, for any $\{x_1,x_2\}\in E(\overline{K_n} \vee K_1)$, the curvature $\kappa^{\mathrm{LLY}}(x_1,x_2)$ is positive for all $n$ by Corollary \ref{cor:specialjoingraphs}\eqref{cor:specialjoingraphs2}.
{Since $C_n \vee K_1$ contains $\overline{K_n} \vee K_1$ as a subgraph, one might guess that the curvature $\kappa^{\mathrm{LLY}}(x_1,x_2)$ should be larger in $C_n \vee K_1$ due to the presence of more paths. However, the computation contradicts this naive expectation.}
\end{example}

\section*{Acknowledgments}
The first author was supported by the National Research Foundation of Korea (NRF) grant funded by the Korean government (MSIT) (No.~RS-2024-00414849).
The second author was supported by the Basic Science Research Program through the National Research Foundation of Korea (NRF) funded by the Ministry of Education (2018R1D1A1B05048450).
The third author was supported by the National Research Foundation of Korea (NRF) grant funded by the Korean government (MSIT) (No. RS-2024-00339854).
\bibliographystyle{amsplain}
\bibliography{reference.bib}

\end{document}